\documentclass[12pt,3p]{elsarticle}
\usepackage{amsmath, amsfonts}
\usepackage{amsthm}
\usepackage{enumitem}
\usepackage{xcolor}
\usepackage{mathtools}
\usepackage{amssymb}
\usepackage{hyperref}
\usepackage{etoolbox}
\usepackage{tgtermes}
\DeclareMathOperator{\Trig}{Trig}

\DeclareMathOperator{\loc}{loc}

\patchcmd{\thmhead}{(#3)}{#3}{}{}
\usepackage{subcaption}
\usepackage{graphicx}
\usepackage{natbib}
\numberwithin{equation}{section}
\makeatletter
\def\ps@pprintTitle{%
	\let\@oddhead\@empty
	\let\@evenhead\@empty
	\def\@oddfoot{}%
	\let\@evenfoot\@oddfoot}
\makeatother

\def\XXint#1#2#3{{\setbox0=\hbox{$#1{#2#3}{\int}$ }
		\vcenter{\hbox{$#2#3$ }}\kern-.6\wd0}}

\newtheorem{theorem}{Theorem}[section]
\newtheorem{lemma}[theorem]{Lemma}

\theoremstyle{definition}
\newtheorem{definition}[theorem]{Definition}

\theoremstyle{remark}
\newtheorem{remark}[theorem]{Remark}
\newif\ifdraft
\draftfalse
\usepackage{graphicx}
\begin{document}
	
	\begin{frontmatter}
		
		
		
		
		\title{Bloch Wave Homogenization of Quasiperiodic Media}		
		\author{Sivaji Ganesh Sista\corref{cor1}}
		\ead{sivaji.ganesh@iitb.ac.in}
		\cortext[cor1]{Corresponding author}
		
		\author{Vivek Tewary\corref{cor2}}
		\ead{vivekt@iitb.ac.in}
		\address{Department of Mathematics, Indian Institute of Technology Bombay, Powai, Mumbai, 400076, India.\\\vspace{0.5cm}\today}
		\begin{abstract}
			Quasiperiodic media is a class of almost periodic media which is generated from periodic media through a ``cut and project" procedure. Bloch waves are typically defined through a direct integral decomposition of periodic operators. A suitable direct integral decomposition is not available for almost periodic operators. To remedy this, we lift an almost periodic operator to a degenerate periodic operator in higher dimensions. Approximate Bloch waves are obtained for a regularized version of the degenerate equation. Homogenized coefficients for quasiperiodic media are determined in terms of the first Bloch eigenvalue of the regularized lifted equation. A notion of quasiperiodic Bloch transform is defined and employed to obtain homogenization limit for an equation with highly oscillating quasiperiodic coefficients.
		\end{abstract}
 		
	
	\end{frontmatter}

\section{Introduction}
In this chapter, we will perform Bloch wave homogenization of the following equation with highly oscillatory quasiperiodic coefficients:
\begin{align}
-\nabla\cdot A\left(\frac{x}{\varepsilon}\right)\nabla u^\varepsilon(x)&=f\mbox{ in }\Omega\nonumber\\
u^\varepsilon&=0\mbox{ on }\partial\Omega,
\end{align} where $\Omega\subseteq\mathbb{R}^d$ is a bounded domain. 
Bloch wave homogenization is a framework developed by Conca and Vanninathan~\cite{Conca1997} for obtaining qualitative as well as quantitative results in periodic homogenization. Further, Bloch decomposition has been employed by Birman and Suslina~\cite{BirmanSuslina2003} to obtain order-sharp estimates for systems in the theory of homogenization with minimal regularity requirements. Let $M$ be an integer such that $M>d$ and let $Q=[0,2\pi)^M$ denote a parametrization of the $M$-dimensional torus $\mathbb{T}^M$. We make the following assumptions on the coefficient matrix $A=(a_{kl})_{k,l=1}^d$:
\begin{enumerate}[label={(H\arabic*)}] 
	\item\label{H1} The coefficients $a_{kl}$ are smooth, bounded real-valued functions defined on $\mathbb{R}^d$.
	\item\label{H2} The coefficient matrix $A$ is quasiperiodic, i.e., there exists a matrix $B$ with amooth $Q$-periodic entries and a constant $M\times d$ matrix $\Lambda$ such that $A=B\circ\Lambda$, viz., \begin{align*}
	\forall\,x\in\mathbb{R}^d\quad a_{kl}(x)=b_{kl}(\Lambda x),
	\end{align*} where 
	\begin{align}\label{irrationalslope}
	\Lambda^T p\neq 0 \mbox{ for  non-zero } p\in\mathbb{Z}^M.\end{align}
	\item\label{H3} The matrix $A$ is symmetric, i.e., $a_{kl}=a_{lk}$.
	\item\label{H4} The matrix $A$ is coercive, i.e., there is a positive real number $\alpha>0$ such that for all $v\in\mathbb{R}^d$ and a.e. $x\in\mathbb{R}^d$, we have \begin{align*}
	\langle A(x)v,v\rangle\geq \alpha ||v||^2.
	\end{align*} Therefore, the matrix $B$ such that $B(\Lambda x)=A(x)$ is coercive on $\Lambda \mathbb{R}^d$.
\end{enumerate}

\begin{remark}\leavevmode
	\begin{enumerate}
		\item The assumption of smoothness on the entries of $A$ is not essential. The approach of this chapter demands taking trace of solutions on lower dimensional manifolds. We only require as much smoothness as would guarantee twice continuous differentiability of the solutions.
		\item The assumption~\eqref{irrationalslope} implies that the continuous and periodic matrix $B$ is uniquely determined from its values on $\Lambda \mathbb{R}^d$. Hence, coercivity of $B$ on $\mathbb{R}^M$ follows from that of $A$. See~\cite{Shubin78} for details.
	\end{enumerate}
\end{remark}

The class of quasiperiodic functions is a subclass of almost periodic functions. For $K=\mathbb{R}\mbox{ or }\mathbb{C}$, let $\text{Trig}(\mathbb{R}^d;K)$ denote the set of all $K$-valued trigonometric polynomials. Recall that the completion of $\text{Trig}(\mathbb{R}^d;K)$ in norm of uniform convergence results in a Banach space called the space of all Bohr almost periodic functions denoted as $AP(\mathbb{R}^d)$. Further, in $L^p_\text{loc}(\mathbb{R}^d)$, one can define a seminorm \begin{align*}
||f||_{B^p}\coloneqq \left(\limsup_{R\to\infty}\frac{1}{R^d}\int_{\left[-\frac{R}{2},\frac{R}{2}\right)^d}|f(y)|^p\,dy\right)^{1/p}.
\end{align*} For $1\leq p<\infty$, the completion of $\text{Trig}(\mathbb{R}^d;K)$ in this seminorm results in the Besicovitch space of almost periodic functions $B^p(\mathbb{R}^d)$. Given a Besicovitch almost periodic function $g$, one can define the notion of mean value \begin{align*}
\mathcal{M}(g)\coloneqq \lim_{R\to\infty}\frac{1}{R^d}\int_{\left[-\frac{R}{2},\frac{R}{2}\right)^d}g(y)\,dy. 
\end{align*} For each $g\in B^p(\mathbb{R}^d)$, we can associate a formal Fourier series $\displaystyle g\sim \sum_{\xi\in\mathbb{R}^d}\widehat{g}(\xi)e^{ix\cdot\xi}$, whose exponents are those vectors $\xi\in\mathbb{R}^d$ such that $\mathcal{M}(g\cdot\exp(ix\cdot\xi))\neq 0$. These exponents or frequencies are denoted by $\exp(g)$ and the $\mathbb{Z}$-module generated by $\exp(g)$ is called as the frequency module of $g$ and denoted by $Mod(g)$. A {\it quasiperiodic} function may also be defined as an almost periodic function whose frequency module is finitely generated (See~\ref{coincide} in Remark~\ref{Remark1}). Trigonometric polynomials are the most common example of quasiperiodic functions. One may conclude from this definition that any quasiperiodic function may be lifted through a {\it winding matrix} $\Lambda$ to a periodic function on a higher dimensional torus. The space of all periodic $L^2$ functions in the higher dimension will be denoted interchangeably by $L^2_\sharp(Q)$ or $L^2(\mathbb{T}^M)$. The space $L^2_\sharp(Q)$ is also defined as the closure of $C^\infty_\sharp(Q)$ functions in $L^2(Q)$ norm. Similarly, for $s\in\mathbb{R}$, we may define $H^s_\sharp(Q)$ or $H^s(\mathbb{T}^M)$ as the space of all periodic distributions for which the norm $||u||_{H^s}=\left(\sum_{n\in\mathbb{Z}^M}(1+|n|^2)^s|\widehat{u}(n)|^2 \right)^{1/2}$ is finite.

\begin{remark}\label{Remark1}\leavevmode
	\begin{enumerate}
		\item The assumption~\eqref{irrationalslope} makes sure that the mean value of the quasiperiodic matrix $A$ can be written as the mean value of the periodic matrix $B$ on $Q$. A proof of this fact may be found in~\cite{Shubin78}. The equality of the two mean values is used in Section~\ref{characterization} for the characterization of homogenized tensor of quasiperiodic media.
		\item\label{coincide} We have given two seemingly disparate definitions of quasiperiodic functions, one as restriction of periodic functions to lower dimensional planes and second through the frequency module. Indeed, the two definitions are equivalent and the proof may be found in~\cite{Blot2001} for different classes of almost periodic functions. Let $\Gamma\subseteq\mathbb{R}^d$ be a finitely generated $\mathbb{Z}$-module. Denote by $B^2_\Gamma(\mathbb{R}^d)$ (respectively $AP_\Gamma(\mathbb{R}^d)$) the subspace of $B^2(\mathbb{R}^d)$ (respectively $AP(\mathbb{R}^d)$) containing functions whose frequencies belong to $\Gamma$. Then, $B^2_\Gamma(\mathbb{R}^d)$ (respectively $AP_\Gamma(\mathbb{R}^d)$) is isometrically isomorphic to $L^2(\mathbb{T}^N)$ (respectively $C(\mathbb{T}^N)$) for some $N>d$. 
		\item A simple example of a quasiperiodic function is $g(x)=\sin(x)+\sin(\sqrt{2}x)$ which admits a periodic embedding of the form $\tilde{g}(x,y)=\sin(x)+\sin(y)$ with $\Lambda=(1\quad\sqrt{2})^T$. One may wonder how to obtain a common matrix $\Lambda$ for a collection of functions such as in the case of the entries of a quasiperiodic matrix. This is not too difficult either, as illustrated in the following example. Consider the quasiperiodic matrix 
		\begin{align*}A=
		\begin{pmatrix}
		\sin(x)+\sin(\sqrt{2}x) & \cos(\sqrt{2}x) \\ 
		\cos(\sqrt{2}x) & \cos(\sqrt{3}x)
		\end{pmatrix}, x\in\mathbb{R}.
		\end{align*} The matrix $A$ admits the following periodic embedding.
		\begin{align*}B=
		\begin{pmatrix}
		\sin(x)+\sin(y) & \cos(y) \\ 
		\cos(y) & \cos(z)
		\end{pmatrix}, (x,y,z)\in\mathbb{R}^3,
		\end{align*} with $\Lambda=(1\quad\sqrt{2}\quad\sqrt{3})^T$.
		\item The assumption~\eqref{irrationalslope} is a qualitative version of Kozlov's small divisors condition which we recall below.	
		\begin{definition}\label{Kozlov}
			A quasiperiodic function $f:\mathbb{R}^d\to\mathbb{R}$ is said to satisfy the {\bf Kozlov condition} if 
			
			(1) there exist a function $F:\mathbb{R}^M\to\mathbb{R}$ and $M(=m_1+m_2+\ldots+m_d)$ numbers
			
			 $\beta_1^1,\beta_1^2,\ldots,\beta_1^{m_1},\beta_2^{1},\ldots,\beta_2^{m_2},\ldots,\beta_d^1,\ldots,\beta_d^{m_d}\in\mathbb{R}$ such that $$f(x)=F\left(\beta_1^1x_1,\beta_1^2x_1,\ldots,\beta_1^{m_1}x_1,\beta_2^{1}x_2,\ldots,\beta_2^{m_2}x_2,\ldots,\beta_d^1x_d,\ldots,\beta_d^{m_d}x_d\right).$$
			
			(2) For each $1\leq i\leq d$, $\beta_i\coloneqq (\beta_i^1,\beta_i^2,\ldots,\beta_i^{m_i})$ is linearly independent over $\mathbb{Z}$.
			
			(3) There exist $C>0$ and $\tau>0$ such that for each $i=1,2,\ldots,d$ such that $m_i\geq2$, \begin{align}
			|n\cdot\beta_i|\geq \frac{C}{|n|^\tau}, \mbox{ for all } n\in\mathbb{Z}^{m_i}\setminus\{0\}.
			\end{align}
		\end{definition}
	\end{enumerate}
\end{remark}

A standard method to solve equations with quasiperiodic coefficients is to propose and solve an equation in higher dimensions whose solutions when suitably restricted to $\mathbb{R}^d$ solve the original equation~\cite{Kozlov78,GloriaHabibi2016,Blanc2015,Wellander2018}. Such a procedure necessitates an assumption on coefficients to be at least continuous since restriction of functions to lower dimensional surfaces requires some smoothness. A second difficulty results from the fact that the equation posed in the higher dimension is typically degenerate or non-elliptic. In order to define a suitable notion of Bloch waves, we regularize the degenerate equation in higher dimension. The homogenized tensor for quasiperiodic media is found to be equal to the limit of the Hessian of first Bloch eigenvalue of the regularized degenerate operator as the regularization parameter tends to zero. Further, we define a notion of quasiperiodic Bloch transform to aid us in the passage to the homogenization limit.

We note here that the study of almost periodic homogenization was initiated by Kozlov~\cite{Kozlov78} who also obtained a rate of convergence for quasiperiodic media satisfying a small divisors condition called the Kozlov condition. A widely known example of quasiperiodic media is quasicrystals~\cite{Schechtman84}. Quasicrystals are ordered structures without periodicity. They may be thought of as periodic crystals in higher dimensions that are projected to lower dimensions through a ``cut and project'' procedure. Quasicrystals have unique thermal and electrical conductivity properties with many potential industrial and household applications, such as adhesion and friction resistant agents, composite materials~\cite{dubois2012properties}. The mathematical structure of quasicrystals had already been anticipated in the works of Bohr~\cite{Bohr47}, Besicovitch~\cite{Besicovitch55}, and Meyer~\cite{Meyer1995}.

The plan of the chapter is as follows: In Section~\ref{degenerate}, we introduce the degenerate periodic equation in $\mathbb{R}^M$ and its regularized version for which we obtain approximate Bloch waves. In Section~\ref{Blochwaves}, we prove the existence of the regularized Bloch waves. In Section~\ref{perturbation}, we apply Kato-Rellich theorem to obtain analytic branch of the regularized Bloch waves and Bloch eigenvalue. In Section~\ref{cellproblemquasiperiodic}, we recall the cell problem for almost periodic media and the cell problem for the degenerate periodic operator in higher dimensions. In Section~\ref{characterization}, we obtain the homogenized tensor for the quasiperiodic media as a limit of the first regularized Bloch eigenvalue. In Section~\ref{quasibloch}, we introduce a notion of quasiperiodic Bloch transform. Finally, in Section~\ref{homotheorem}, we obtain the homogenization theorem for quasiperiodic media by using the quasiperiodic Bloch transform.

\section{Degenerate operator in $\mathbb{R}^M$}\label{degenerate} The Bloch wave method in homogenization is a spectral method. Bloch waves are solutions to the Bloch spectral problem which is a parametrized eigenvalue problem. While the details of Bloch wave method can be found in~\cite{Conca1997}, the main feature of this method is the existence of a ``ground state'' for the periodic operator, which is facilitated by the direct integral decomposition of the periodic operator. In the case of a quasiperiodic operator, one may not have a ground state but we show the existence of an approximate ground state. To begin with, we shall pose a Bloch spectral problem for the quasiperiodic operator. Let $\displaystyle Y^{'}\coloneqq \left[-\frac{1}{2},\frac{1}{2} \right)^d$, then we seek quasiperiodic solutions to the following Bloch spectral problem for the quasiperiodic operator $\mathcal{A}=-\nabla\cdot(A\nabla)$ 
\begin{align}\label{quasiBloch}
-(\nabla+i\eta)\cdot A(\nabla+i\eta)\phi=\lambda\phi\mbox{ in }\mathbb{R}^d.
\end{align}

The problem above is typically solved for periodic $A$, in which case, the solutions are called Bloch waves. However, the matrix $A$ is quasiperiodic, and it is not clear whether quasiperiodic solutions to~\eqref{quasiBloch} exist. Therefore, we propose to lift the operator $\mathcal{A}$ to a periodic operator in $\mathbb{R}^M$, for which a functional analytic formalism is available. The mapping $x\mapsto \Lambda x\in\mathbb{R}^M$ lifts the operator $\mathcal{A}$ to the periodic but degenerate operator in $\mathbb{R}^M$ given by
\begin{align}
\mathcal{C}\coloneqq -\Lambda^T\nabla_y\cdot B\Lambda^T\nabla_y.
\end{align}
Let us denote $\Lambda^T\nabla_y$ by $D$, then operator $\mathcal{C}$ is written as $-D\cdot BD$. The operator $\mathcal{C}$ may also be written as $-\nabla_y\cdot C\nabla_y$ where the matrix $C=\Lambda B\Lambda^T$. Note that $C$ is non-coercive. 

The Bloch eigenvalue problem given by~\eqref{quasiBloch} is lifted to the following problem: For $\eta\in Y^{'}$, find $\phi(\eta)\in H^1_\sharp(Q)$ such that
\begin{align}\label{Blochproblem}
\mathcal{C}(\eta)\phi(\eta)\coloneqq -(D+i\eta)\cdot B(D+i\eta)\phi(\eta)=\lambda(\eta)\phi(\eta).
\end{align} We note here that due to the degeneracy of operator $\mathcal{C}(\eta)$, we cannot seek Lax-Milgram solutions to this equation in $H^1_\sharp(Q)$. To remedy this situation, inspired by~\cite{Blanc2015}, we regularize~\eqref{Blochproblem} as follows. For $\eta\in Y^{'}$ and $0<\delta<1$, find $\phi^\delta(\eta)\in H^1_\sharp(Q)$ such that
\begin{align}\label{deltaBlochproblem}
\mathcal{C}^\delta(\eta)\phi^\delta(\eta)\coloneqq -(D+i\eta)\cdot B(D+i\eta)\phi^\delta(\eta)+\delta\Delta\phi^\delta(\eta)=\lambda^\delta(\eta)\phi^\delta(\eta).
\end{align}

The solutions $\phi^\delta$ to~\eqref{deltaBlochproblem} shall be called regularized Bloch waves and $\lambda^\delta$ will be called regularized Bloch eigenvalues.

\begin{remark}
	\leavevmode
	\begin{enumerate}
		\item In homogenization, one often assumes the basic periodicity cell to be rectangular for convenience. However, more general periodicity cells in the shape of a parallelopiped may be considered through a change of coordinates. Under the change of coordinates, the rectangular cell becomes a parallelopiped and an operator of the form $-\nabla\cdot A\nabla$ becomes $-\nabla\cdot(PAP^{-1})\nabla$. In a similar fashion, the transformation $\Lambda$ converts the operator $-\nabla_x\cdot A\nabla_x$ into the operator $-\nabla_y\cdot(\Lambda B\Lambda^T)\nabla_y$. Unlike $PAP^{-1}$, the matrix $\Lambda B\Lambda^T$ is non-invertible since $\Lambda$ is a transformation between spaces of different dimensions.
		\item It is instructive to compare quasiperiodic structures with laminates. Quasiperiodic media admit embeddings in higher-dimensions which are periodic and non-homogeneous in all directions. On the other hand, laminated materials are periodic structures which are homogeneous in some directions. Further, the operator with quasiperiodic coefficients has a degenerate embedding in higher dimensions, viz., it is non-elliptic in certain directions. On the other hand, the operator modelling laminates are elliptic in all directions.
		\item The regularization may be thought of as the addition of complementary directions to the quasicrystal which is produced by ``cutting'' a periodic crystal in certain ``irrational'' directions and then projecting to lower dimensions.
		\item In contrast with~\eqref{deltaBlochproblem}, it is standard to take the quasimomentum parameter $\eta$ in $\mathbb{R}^M$ and to seek the regularized Bloch eigenvalues corresponding to the periodic operator given by $-\nabla_y\cdot(\Lambda B\Lambda^T+\delta I)\nabla_y$ where $I$ is the $M\times M$ identity matrix. However, we have chosen the quasimomentum parameter $\eta$ in $\mathbb{R}^d$ and we have not introduced a shift in the regularized term $\delta\Delta$. This simplifies the presentation considerably.
	\end{enumerate}
\end{remark}
	
\section{Regularized Bloch waves}\label{Blochwaves}
In what follows, we shall prove that 
	\begin{enumerate}
		\item There exists $C_*$ such that for all $\displaystyle\eta\in Y^{'}$, the bilinear form generated by the operator $\mathcal{C}^\delta(\eta)+C_*I$ is elliptic on $H^1_\sharp(Q)$ where $I$ denotes the identity operator on $L^2_\sharp(Q)$. This will allow us to prove invertibility of $\mathcal{C}^\delta(\eta)+C_*I$. 
		\item By Rellich compactness theorem, we will prove compactness of the inverse of $\mathcal{C}^\delta(\eta)+C_*I$ in $L^2_\sharp(Q)$. This will prove the existence of regularized Bloch eigenvalues and Bloch eigenfunctions.
		\item An application of the perturbation theory will provide us with smoothness of regularized Bloch eigenvalues and Bloch waves with respect to $\eta$ near $\eta=0$. 
	\end{enumerate} 

For the bilinear form $a^\delta[\eta](\cdot,\cdot)$ defined on $H^1_\sharp(Q)\times H^1_\sharp(Q)$ by \begin{align}
	a^\delta[\eta](u,v)\coloneqq \int_Q B (D+i\eta)u\cdot\overline{(D+i\eta)v}\,dy+\delta\int_Q \nabla_y u\cdot\overline{\nabla_y v}\,dy,
\end{align} we have the following G\aa rding-type inequality whose proof is simple and is omitted.

\begin{lemma}\label{coercivityfortranslate}
There exist positive real numbers $C_*$ and $C^*$ not depending on $\delta$ and $\eta$ such that for all $u\in H^1_\sharp(Q)$ and all $\eta\in Y^{'}$, we have	
\begin{align}\label{coercivity}
a[\eta](u,u)+C_*||u||^2_{L^2_\sharp(Q)}\geq \delta||\nabla_y u||^2_{L^2_\sharp(Q)}+C^*||Du||_{L^2_\sharp(Q)}. 
\end{align}
\end{lemma}

The above lemma shows that for every $\eta\in Y^{'}$ the operator $\mathcal{C}^\delta(\eta)+C_*I$ is elliptic on $H^1_\sharp(Q)$. Hence, for $f\in L^2_\sharp(Q)$, this shows that $\mathcal{C}^\delta(\eta)u+C_*u=f$ is solvable and the solution is in $H^1_\sharp(Q)$. As a result, the solution operator $S(\eta)$ is continuous from $L^2_\sharp(Q)$ to $H^1_\sharp(Q)$. Since the space $H^1_\sharp(Q)$ is compactly embedded in $L^2_\sharp(Q)$, $S(\eta)$ is a self-adjoint compact operator on $L^2_\sharp(Q)$. Therefore, by an application of the spectral theorem for self-adjoint compact operators, for every $\eta\in Y^{'}$ we obtain an increasing sequence of eigenvalues of $\mathcal{C}^\delta(\eta)+C_*I$ and the corresponding eigenfunctions form an orthonormal basis of $L^2_\sharp(Q)$. However, note that both the operators $\mathcal{C}^\delta(\eta)$ and $\mathcal{C}^\delta(\eta)+C_*I$ have the same eigenfunctions but each eigenvalue of the two operators differ by $C_*$. We shall denote the eigenvalues and eigenfunctions of the operator $\mathcal{C}^\delta(\eta)$ by $\eta\to(\lambda^\delta_m(\eta),\phi^\delta_m(\cdot,\eta))$. Note that due to the regularity of the coefficients, the eigenfunctions are $C^\infty$ functions of $y\in Q$. All of these developments are recorded in the theorem below.

\begin{theorem}\label{smoothnessofeigenfunc}
	The regularized Bloch eigenvalue problem~\eqref{deltaBlochproblem} admits a countable sequence of eigenvalues and corresponding eigenfunctions in the space $H^1_\sharp(Q)$. Further, the eigenfunctions $\phi_m(y,\eta)$ are $C^\infty$ functions of $y\in Q$.
\end{theorem}

\begin{proof}
  We have already proved the existence of the eigenvalues and eigenfunctions for the problem~\eqref{deltaBlochproblem}. Regularity of the eigenfunctions follows from the standard elliptic regularity theory~\cite{Ladyzhenskaya68}.
\end{proof}

\begin{remark}
	In~\ref{H1}, we assume the coefficient matrix $A$ to be smooth. However, we do not require this much regularity. We only require as much smoothness on the coefficient matrix that would ensure that the Bloch eigenfunctions are twice continuously differentiable.
\end{remark}

\section{Regularity of the ground state}\label{perturbation} In the sequel, differentiability properties of regularized Bloch eigenvalues and regularized Bloch eigenfunctions with respect to the dual parameter $\eta\in Y^{'}$ are required. For this purpose, we have Kato-Rellich theorem~\cite{Kato1995} which guarantees analyticity of parametrized eigenvalues and eigenfunctions corresponding to analytic family of operators near a point at which the eigenvalue is simple. Indeed, we will prove the following theorem.
 
\begin{theorem}\label{analytic}
	There is a small ball $U^\delta \coloneqq B_{\theta_{\delta}}(0) \coloneqq\{\eta\in Y^{'}:|\eta|<\theta_{\delta}\}$ such that
	\begin{enumerate}
		\item The first regularized Bloch eigenvalue $\eta\to \lambda^\delta_1(\eta)$ is analytic for $\eta\in U^\delta$.
		\item There is a choice of corresponding eigenfunctions $\phi^\delta_1(\cdot,\eta)$ such that $\eta\in U^\delta\to \phi_1^\delta(\cdot,\eta)\in H^1_\sharp(Q)$ is analytic.
	\end{enumerate}
\end{theorem}

The proof will require the Kato-Rellich theorem which we will state below for completeness. The theorem as stated in~\cite{Reed1978} is for a single parameter, however the theorem is also true for multiple parameters with the assumption of simplicity (See Supplement of~\cite{Baumgartel1985}).

\begin{theorem}{(Kato-Rellich)}
	Let $D(\tilde{\eta})$ be a self-adjoint holomorphic family of type (B) defined for $\tilde{\eta}$ in an open set in $\mathbb{C}^M$. Further let $\lambda_0=0$ be an isolated eigenvalue of $D(0)$ that is algebraically simple.  Then there exists a neighborhood $R_0\subseteq \mathbb{C}^M$ containing $0$ such that for $\tilde{\eta}\in R_0$, the following holds:
	\begin{enumerate}
		\item There is exactly one point $\lambda(\tilde{\eta})$ of $\sigma(D(\tilde{\eta}))$ near $\lambda_0=0$. Also, $\lambda(\tilde{\eta})$ is isolated and algebraically simple. Moreover, $\lambda(\tilde{\eta})$ is an analytic function of $\tilde{\eta}$.
		\item There is an associated eigenfunction $\phi(\tilde{\eta})$ depending analytically on $\tilde{\eta}$ with values in $H^1_\sharp(Q)$.
	\end{enumerate}
\end{theorem}

In order to prove Theorem~\ref{analytic}, we need to complexify the shifted operator $\mathcal{C}^\delta(\eta)$ before verifying the hypothesis of Kato-Rellich Theorem. 

\begin{proof}{(Proof of Theorem~\ref{analytic})}
 \paragraph{(i) Complexification of $\mathcal{C}^\delta(\eta)$} The form $a[\eta](\cdot,\cdot)$ is associated with the operator $\mathcal{C}^\delta(\eta)$.  We define its complexification as 
 $$t(\tilde{\eta})=\int_Q B(D+i\sigma+\tau)u\cdot(D-i\sigma+\tau)\overline{u}\,dy+\delta\int_Q \nabla_y u\cdot\overline{\nabla_y u}\,dy$$ 
 for $\tilde{\eta}\in R$ where \begin{equation*}R\coloneqq\{\tilde{\eta}\in\mathbb{C}^M:\tilde{\eta}=\sigma+i\tau, \sigma,\tau\in\mathbb{R}^M, |\sigma|<1/2,|\tau|<1/2\}.\end{equation*}
\paragraph{(ii) the form $t(\tilde{\eta})$ is sectorial}
We have \begin{align*}
	t(\tilde{\eta})&=\int_Q B(D+i\sigma+\tau)u\cdot(D-i\sigma+\tau)\overline{u}\,dy +\delta\int_Q \nabla_y u\cdot\overline{\nabla_y u}\,dy\\
	&=\int_Q B(D+i\sigma) u\cdot(D-i\sigma)\overline{u}\,dy+\delta\int_Q \nabla_y u\cdot\overline{\nabla_y u}\,dy-\int_Q B(\tau u)\cdot D\overline{u}\,dy\\
	&\qquad + \int_Q B D u\cdot(\tau\overline{u})\, dy-\int_Q B\tau u\cdot\tau\overline{u}\,dy+i \int_Q B\sigma u\cdot\tau\overline{u}\,dy+i \int_Q B\tau u\cdot\sigma \overline{u}\,dy.
\end{align*} From above, it is easy to write separately the real and imaginary parts of the form $t(\tilde{\eta})$.
\begin{align}
	\Re t(\tilde{\eta})[u]=\int_Q B(D+i\sigma) u\cdot(D-i\sigma)\overline{u}\,dy +\delta\int_Q \nabla_y u\cdot\overline{\nabla_y u}\,dy -\int_Q B\tau u\cdot\tau\overline{u}\,dy,
\end{align}
\begin{align}
	\Im\,t(\tilde{\eta})[u]= \int_Q B\sigma u\cdot\tau\overline{u}\,dy + \int_Q B\tau u\cdot\sigma \overline{u}\,dy + \Im \int_Q B D u\cdot \tau \overline{u}\,dy.
\end{align} 
For the real part, we can readily obtain the following estimate:
\begin{align}\label{sectoriality2}
  \Re t(\tilde{\eta})[u] + C_5 ||u||^2_{L^2_\sharp(Q)}\geq \frac{\alpha}{2}\left(||u||^2_{L^2_\sharp(Q)}+||Du||^2_{L^2_\sharp(Q)}\right) + \delta||\nabla_y u||^2_{L^2_\sharp(Q)}.	
\end{align}
Let us define the new form $\tilde{t}(\tilde{\eta})$ by $\tilde{t}(\tilde{\eta})[u,v]=t(\tilde{\eta})[u,v]+(C_5+C_6)(u,v)_{L^2_\sharp(Q)}$, for which it holds that 
\begin{align*}
	\Re \tilde{t}(\tilde{\eta})[u]\geq \frac{\alpha}{2}\left(||u||^2_{L^2_\sharp(Q)}+||Du||^2_{L^2_\sharp(Q)}\right) +\delta||\nabla_y u||^2_{L^2_\sharp(Q)}+C_6||u||^2_{L^2_\sharp(Q)}.
\end{align*}
Also, the imaginary part of $\tilde{t}(\tilde{\eta})$ can be estimated as follows:
\begin{align*}
	\Im \tilde{t}(\tilde{\eta})[u]&\leq C_7 ||u||^2_{L^2_\sharp(Q)}+C_8 ||Du||^2_{L^2_\sharp(Q)}\\
	&\stackrel{C_7=C_6C_9,2C_8=\alpha C_9}{=}C_9\left(C_6||u||^2_{L^2_\sharp(Q)}+\frac{\alpha}{2}||Du||^2_{L^2_\sharp(Q)}\right)\\
	&\leq C_9\left(\Re \tilde{t}(\tilde{\eta})[u]-\frac{\alpha}{2}||u||^2_{L^2_\sharp(Q)}\right).
\end{align*} This shows that $\tilde{t}(\tilde{\eta})$ is sectorial. However, sectoriality is invariant under translations by scalar multiple of identity operator in $L^2_\sharp(Q)$, therefore the form $t(\tilde{\eta})$ is also sectorial.
\paragraph{(iii) The form $t(\tilde{\eta})$ is closed}
Suppose that $u_n\stackrel{t}{\to}u$. This means that $u_n\to u$ in $L^2_\sharp(Q)$ and $t(\tilde{\eta})[u_n-u_m]\to 0$. As a consequence, $\Re t(\tilde{\eta})[u_n-u_m]\to 0$. By~\eqref{sectoriality2}, $||u_n-u_m||_{H^1_\sharp(Q)}\to 0$, i.e., $(u_n)$ is Cauchy in $H^1_\sharp(Q)$. Therefore, there exists $v\in H^1_\sharp(Q)$ such that $u_n\to v$ in $H^1_\sharp(Q)$. Due to uniqueness of limit in $L^2_\sharp(Q)$, $v=u$. Therefore, the form is closed.
\paragraph{(iv) The form $t(\tilde{\eta})$ is holomorphic} The holomorphy of $t$ is an easy consequence of the fact that $t$ is a quadratic polynomial in $\eta$.
\paragraph{(v) $0$ is an isolated eigenvalue} Zero is an eigenvalue because constants belong to the kernel of $\mathcal{C}^\delta(0)=-\nabla_y\cdot(\Lambda B\Lambda^T+\delta I)\nabla_y$. We proved using Lemma~\ref{coercivityfortranslate} that $\mathcal{C}^\delta(0)+C_*I$ has compact resolvent. Also, $C_*$ is an eigenvalue of $\mathcal{C}^\delta(0)+C_*I$. Therefore, $C_*^{-1}$ is an eigenvalue of $(\mathcal{C}^\delta(0)+C_*I)^{-1}$ and $C_*^{-1}$ is isolated. Hence, zero is an isolated point of the spectrum of $\mathcal{C}^\delta(0)$.
\paragraph{(vi) $0$ is a geometrically simple eigenvalue} Denote by $ker\,\mathcal{C}^\delta(0)$ the nullspace of operator $\mathcal{C}^\delta(0)$. Let $v\in ker\,\mathcal{C}^\delta(0)$, then $\int_Q (\Lambda B\Lambda^T+\delta I)\nabla_y v\cdot\nabla_y v\,dy=0$. Due to the coercivity of the matrix $(\Lambda B\Lambda^T+\delta I)$, we obtain $||\nabla_y v||_{L^2_\sharp(Q)}=0$. Hence, $v$ is a constant. This shows that the eigenspace corresponding to eigenvalue $0$ is spanned by constants, therefore, it is one-dimensional.
\paragraph{(vii) $0$ is an algebraically simple eigenvalue} Suppose that $v\in H^1_\sharp(Q)$ such that $\mathcal{C}^\delta(0)^2v=0$, i.e., $\mathcal{C}^\delta(0)v\in ker\,\mathcal{C}^\delta(0)$. This implies that $\mathcal{C}^\delta(0)v=C$ for some generic constant $C$. However, by the compatibility condition for the solvability of this equation, we obtain $C=0$. Therefore, $v\in ker\,\mathcal{C}^\delta(0)$. This shows that the eigenvalue $0$ is algebraically simple.
\end{proof}

\section{Cell problem for quasiperiodic media}\label{cellproblemquasiperiodic} In this section, we shall recall the cell problem~\cite{Oleinik1982} in the theory of almost periodic homogenization as well as the cell problem for the degenerate periodic operator in higher dimensions~\cite{Kozlov78} for quasiperiodic media.

Let $e_l$ be the unit vector in $\mathbb{R}^d$ with $1$ in the $l^{th}$ place and $0$ elsewhere. For almost periodic media, the cell problem \begin{align}\label{cellproblem}
	-\nabla_x\cdot(A(x)(e_l+\nabla_x w_l))=0
\end{align} is not solvable in the space of almost periodic functions. Hence, an abstract setup is required which is explained below. Let $S=\{\nabla_x \phi:\phi\in \Trig(\mathbb{R}^d;\mathbb{R}) \}$. This is a subspace of $(B^2(\mathbb{R}^d))^d$. We shall call the closure of $S$ in $(B^2(\mathbb{R}^d))^d$ as $W$. For the matrix $A$, we define a bilinear form on $W$ by \begin{align*}
	a(w^1,w^2)=\sum_{j,k=1}^d\mathcal{M}(a_{jk}w^1_jw^2_k),
\end{align*} where $w^1=(w_1^1,w_2^1,\ldots,w_d^1)$ and $w^2=(w_1^2,w_2^2,\ldots, w_d^2)$. By coercivity of the matrix $A$, the bilinear form is coercive. Also, by boundedness of $A$, the bilinear form is continuous on $W\times W$. We also define the following linear form on $W$: \begin{align*}
    L_l(V)\coloneqq -\sum_{k=1}^d\mathcal{M}(a_{kl})v_k.
\end{align*} Again, by boundedness of matrix $A$, the linear form $L$ is continuous. Hence, Lax-Milgram lemma guarantees a solution to the following problem: Find $N^l\in W$ such that $\forall\, V\in W$, we have \begin{align}\label{abstractcell}
	a(N^l,V)=L_l(V).
\end{align} This is the abstract cell problem for almost periodic homogenization~\cite{Oleinik1982} and the homogenized coefficients are defined as \begin{align}
	q_{kl}^*=\mathcal{M}\left(a_{kl}+\sum_{j=1}^da_{kj}N^l_j \right).
\end{align}
However, in the case of quasiperiodic media, one can also define cell problem in higher dimensions as in~\cite{Kozlov78}. The transformation $x\mapsto\Lambda x$ converts the cell problem in $\mathbb{R}^d$~\eqref{cellproblem} to a cell problem posed in $Q$ for the degenerate periodic operator. \begin{align}\label{cellproblem2}
	-D\cdot B(y)D \psi_l= D\cdot B(y)e_l.
\end{align}
Due to the lack of coercivity, we implement the regularizing trick as in~\cite{Blanc2015}. For $0<\delta<1$, we seek the solution $\psi_l^\delta\in H^1_\sharp(Q)/\mathbb{R}$ to the following equation.
\begin{align}\label{cellproblem3}
	-D\cdot B(y)D \psi^\delta_l-\delta\Delta\psi^\delta_l= D\cdot B(y)e_l.
\end{align}
The solution satisfies the a priori bound $||D\psi^\delta_l||^2_{L^2_\sharp(Q)}+\delta||\nabla_y \psi^\delta_l||^2_{L^2_\sharp(Q)}\leq C$ for some generic constant $C$. As a consequence, $D\psi_l^\delta$ converges to some function $\chi^l\in (L^2_\sharp(Q))^d$ for a subsequence in the limit $\delta\to 0$. Using the a priori bounds, we can pass to the limit $\delta\to 0$ in the equation~\eqref{cellproblem3} to show that $\chi^l$ solves the equation~\eqref{cellproblem2} in the form
\begin{align}\label{cellproblem4}
-D\cdot B(y)\chi^l= D\cdot B(y)e_l.
\end{align}
By elliptic regularity, $Dw_l^\delta\in H^s_\sharp(Q)$ for all $s>0$. As a consequence, $Dw_l^\delta\in C^\infty(Q)$. Therefore, $\chi^l\in H^s_\sharp(Q)$ for all $s>0$. Again, $\chi^l\in C^\infty(Q)$ and the equation~\eqref{cellproblem4} holds pointwise. Hence, we can restrict equation~\eqref{cellproblem4} to $\mathbb{R}^d$ using the matrix $\Lambda$. Define $N^l(x)=\chi^l(\Lambda x)$, then $N^l$ solves the abstract cell problem~\eqref{abstractcell}. Therefore, the homogenized coefficients can be written in terms of the solution of the lifted cell problem~\eqref{cellproblem2}.

\begin{align}
q_{kl}^*=\mathcal{M}\left(a_{kl}+\sum_{j=1}^da_{kj}N^l_j \right)=\mathcal{M}_Q\left(b_{kl}+\sum_{j=1}^db_{kj}\chi^l_j \right).
\end{align}

Further, we have also shown that if we define the approximate homogenized tensor $A^{\delta,*}=(q^{\delta,*}_{kl})$ as 
\begin{align}\label{approxBloch}
q^{\delta,*}_{kl}=e_k\cdot A^{\delta,*}e_l=\mathcal{M}_Q\left( b_{kl}+e_k\cdot BDw_l^\delta \right),
\end{align} then $A^{\delta,*}\to A^{*}$, where $A^*=(q^*_{kl})$ is defined by
\begin{align}\label{exactBloch}
q^*_{kl}=e_k\cdot A^{*}e_l=\mathcal{M}_Q\left( b_{kl}+e_k\cdot B\chi^l \right).
\end{align}

\begin{lemma}\label{convergenceofapproxhomo}
     The approximate homogenized matrix $q^{\delta,*}_{kl}$ converges to the homogenized matrix $q^*_{kl}$ of quasiperiodic media as defined in~\eqref{exactBloch}.	
\end{lemma}

\section{Characterization of homogenized tensor}\label{characterization} Now, we shall compute derivatives with respect to $\eta$ of the first regularized Bloch eigenvalue and first regularized Bloch eigenfunction at the point $\eta=0$ and identify the homogenized tensor for quasiperiodic media. Note that the regularized Bloch eigenvalues and eigenfunctions are defined as functions of $\eta\in Y^{'}$. The first regularized Bloch eigenfunction satisfies the following problem in $Q$:
\begin{align}\label{Blochproblem3}
	-(D+i\eta)\cdot B(y)(D+i\eta)\phi^\delta_1(y;\eta)-\delta\Delta\phi^\delta_1(y;\eta)=\lambda^\delta_1(\eta)\phi^\delta_1(y;\eta).
\end{align}
We know that $\lambda^\delta_1(0)=0$. For $\eta\in Y^{'}$, recall that $\mathcal{C}^\delta(\eta)= -(D+i\eta)\cdot B(y)(D+i\eta)-\delta\Delta$. In the rest of this section, we will suppress the dependence on $y$ for convenience. For $l=1,2,\ldots,d$, differentiate equation~\eqref{Blochproblem3} with respect to $\eta_l$ to obtain
\begin{align}\label{derivative1}
	\frac{\partial \mathcal{C}^\delta}{\partial \eta_l}(\eta){\phi^\delta_1}(\eta)+\mathcal{C}^\delta(\eta)\frac{\partial {\phi^\delta_1}}{\partial \eta_l}(\eta)={\lambda}^\delta_1(\eta)\frac{\partial {\phi^\delta_1}}{\partial \eta_l}(\eta)+\frac{\partial {\lambda}^\delta_1}{\partial\eta_l}(\eta){\phi^\delta_1}(\eta),
\end{align} where $\displaystyle \frac{\partial \mathcal{C}}{\partial \eta_l}(\eta)=-i D\cdot(Be_l)-ie_l\cdot(B D)+e_l\cdot B\eta+\eta\cdot B e_l$, where $e_l$ is the unit vector in $\mathbb{R}^d$ with $1$ in the $l^{th}$ place and $0$ elsewhere. We multiply~\eqref{derivative1} by $\overline{{\phi^\delta_1}(\eta)}$, take mean value over $Q$ and set $\eta=0$ to get $\displaystyle\frac{\partial{\lambda}^\delta_1}{\partial\eta_l}(0)=0$ for all $l=1,2,\ldots,d$.

On the other hand, if we set $\eta=0$ in~\eqref{derivative1}, we obtain 

\begin{align*}
	\mathcal{C}^\delta(0)\frac{\partial{\phi^\delta_1}}{\partial\eta_l}(0)=-\frac{\partial\mathcal{C}^\delta}{\partial \eta_l}(0){\phi^\delta_1}(0),
\end{align*} or
\begin{align*}
	\left(-D \cdot B(y)D-\delta\Delta\right)\frac{\partial {\phi^\delta_1}}{\partial \eta_l}(0)=D\cdot B(y)e_l i{\phi^\delta_1}(0).
\end{align*} 

Hence, $\displaystyle\psi^\delta_l-\frac{1}{i{\phi^\delta}_1(0)}\frac{\partial\phi^\delta_1}{\partial\eta_l}(0)$ is a constant.

Now, differentiate~\eqref{derivative1} with respect to $\eta_k$ to obtain
\begin{align}\label{derivative2}
	\left(\frac{\partial^2 \mathcal{C}^\delta}{\partial\eta_k\partial\eta_l}(\eta)-\frac{\partial^2 {\lambda}^\delta_1}{\partial\eta_k\partial\eta_l}(\eta) \right)\phi^\delta_1(\eta)+\left(\frac{\partial \mathcal{C}^\delta}{\partial\eta_k}(\eta)-\frac{\partial {\lambda}^\delta_1}{\partial\eta_k}(\eta) \right)\frac{\partial\phi^\delta_1}{\partial\eta_l}(\eta)+\nonumber\\
	\qquad\left(\frac{\partial \mathcal{C}^\delta}{\partial\eta_l}(\eta)-\frac{\partial {\lambda}^\delta_1}{\partial\eta_l}(\eta) \right)\frac{\partial\phi^\delta_1}{\partial\eta_k}(\eta)+\left( \mathcal{C}^\delta(\eta)-\lambda_1^\delta(\eta) \right)\frac{\partial^2\phi_1^\delta}{\partial\eta_l\partial\eta_k}(\eta)=0.
\end{align}
Multiply with $\overline{\phi^\delta_1(\eta)}$, take mean value over $Q$ and set $\eta=0$ to obtain
\begin{align}\label{homogenizedtensor2}
	\frac{1}{2}\frac{\partial^2 {\lambda}^\delta_1}{\partial\eta_k\partial\eta_l}(0)=\mathcal{M}_Q\left(b_{kl}+\frac{1}{2}e_k\cdot B D w_l+\frac{1}{2}e_l\cdot B D w_k\right).
\end{align} Thus, we have proved the following theorem:

\begin{theorem}\label{Hessian}
	The regularized first Bloch eigenvalue and eigenfunction satisfy:
	\begin{enumerate}
		\item $\lambda^\delta_1(0)=0$.
		\item The eigenvalue $\lambda^\delta_1(\eta)$ has a critical point at $\eta=0$, i.e., \begin{align}\frac{\partial \lambda^\delta_1}{\partial \eta_l}(0)=0, \forall l=1,2,\ldots,d.\end{align}
		\item For $l=1,2,\ldots,d,$ the derivative of the eigenvector $(\partial \phi_1^\delta/\partial\eta_l)(0)$ satisfies:
		
		$(\partial \phi_1^\delta/\partial\eta_l)(y;0)-i\phi^\delta_1(y;0)\psi^{\delta}_l(y)$ is a constant in $y$ where $\psi^\delta_l$ solves the cell problem~\eqref{cellproblem3}.
		\item The Hessian of the first Bloch eigenvalue at $\eta=0$ is twice the approximate homogenized matrix $q_{kl}^{\delta,*}$ as defined in~\eqref{approxBloch}, i.e.,
		\begin{align}\label{identificationofapproximatehomo}
		\frac{1}{2}\frac{\partial^2\lambda^\delta_1}{\partial\eta_k\partial\eta_l}(0)=q_{kl}^{\delta,*}		
		\end{align}
	\end{enumerate}
\end{theorem}

\section{Quasiperiodic Bloch transform}\label{quasibloch} We shall normalize $\phi^\delta_1(y;0)$ to be $(2\pi)^{-d/2}$. The Bloch problem at $\epsilon$-scale is given by \begin{align}\label{Blochproblemepsilon}
-(D_{y^{'}}+i\xi)\cdot B(y^{'}/\epsilon)(D_{y^{'}}+i\xi)\phi^{\delta,\epsilon}_1(y^{'};\xi)-\delta\Delta_{y^{'}}\phi^{\delta,\epsilon}_1(y^{'};\xi)={\lambda}^{\delta,\epsilon}_1(\xi){\phi}_1^{\delta,\epsilon}(y^{'};\xi)
\end{align} for $y\in \epsilon Q$ and $\xi\in\epsilon Y^{'}$. Due to the transformation $y=y^{'}/\epsilon$ and $\eta=\epsilon\xi$, we have ${\lambda}_1^{\delta,\epsilon}(\xi)=\epsilon^{-2}{\lambda}_1^\delta(\epsilon\xi)$ and ${\phi}_1^{\delta,\epsilon}(y^{'};\xi)={\phi}^\delta_1(y^{'}/\epsilon;\epsilon\xi)$.
The above equation holds pointwise for $y^{'}\in \epsilon Q$ and is analytic for $\xi\in \epsilon^{-1} U^\delta$. For the purpose of Bloch wave homogenization, we need to restrict the regularized Bloch eigenvalues and eigenfunctions to $\mathbb{R}^d$ using the matrix $\Lambda$. Let us define $\tilde{\phi}_1^{\delta,\epsilon}(x;\xi)\coloneqq\phi^\delta_1(\frac{\Lambda x}{\epsilon};\epsilon\xi)$. Also define $\beta_1^{\delta;\epsilon}(y^{'},\xi)\coloneqq\sqrt{\delta}\Delta_{y^{'}}\phi_1^{\delta,\epsilon}(y^{'};\xi)$ and its restriction $\tilde{\beta}_1^{\delta;\epsilon}(x,\xi)=\sqrt{\delta}\Delta_x\phi_1^{\delta,\epsilon}(\Lambda x;\xi)$, then the restriction of the first regularized Bloch eigenfunction satisfies the following approximate spectral problem in $\mathbb{R}^d$.

\begin{align}\label{projectedeq}
	-(\nabla_x+i\xi)\cdot A\left(\frac{x}{\epsilon}\right)(\nabla_x+i\xi)\tilde{\phi}_1^{\delta,\epsilon}(x;\xi)-\sqrt{\delta}\tilde{\beta}_1^{\delta;\epsilon}(x,\xi)={\lambda}_1^{\delta,\epsilon}(\xi)\tilde{\phi}_1^{\delta,\epsilon}(x;\xi).
\end{align}

We can compare this to our original goal of solving equation~\eqref{quasiBloch} in $\mathbb{R}^d$. Although we could not solve the exact quasiperiodic Bloch spectral problem, we could solve an approximate quasiperiodic Bloch problem using the lifted periodic problem. Interestingly, the functions $\tilde{\phi}_1^{\delta,\epsilon}(x;\xi)$ and $\tilde{\beta}_1^{\delta;\epsilon}(x,\xi)$ are quasiperiodic functions of the first variable.

Now we can define a dominant Bloch coefficient for compactly supported functions in $\mathbb{R}^d$ by employing the first regularized Bloch eigenfunction as follows: Let $g\in L^2(\mathbb{R}^d)$ with compact support, then define
\begin{align}
	\mathcal{B}_1^{\delta,\epsilon} g(\xi)\coloneqq \int_{\mathbb{R}^d} g(x) e^{-ix\cdot\xi}\,\overline{\tilde\phi_1^{\delta,\epsilon}}(x;\xi)\,dx.
\end{align}

For the next section, we need to know the limit of Bloch transform of a sequence of functions as below.

\begin{theorem}
	Let $K\subseteq\mathbb{R}^d$ be a compact set and $(g^\epsilon)$ be a sequence of functions in $L^2(\mathbb{R}^d)$ such that $g^\epsilon=0$ outside $K$. Suppose that $g^\epsilon\rightharpoonup g$ in $L^2(\mathbb{R}^d)$-weak for some function $g\in L^2(\mathbb{R}^d)$. Then it holds that \begin{align*}
          \chi_{\epsilon^{-1}U^\delta}\mathcal{B}_1^{\delta,\epsilon} g^\epsilon\rightharpoonup \widehat{g}
	\end{align*} in $L^2_{\text{loc}}(\mathbb{R}^d_\xi)$-weak, where $\widehat{g}$ denotes the Fourier transform of $g$.
\end{theorem}

\begin{proof}
    The function $\mathcal{B}_1^{\delta,\epsilon} g^\epsilon$ is defined for $\xi\in \epsilon^{-1}Y^{'}$. However, we shall treat it as a function on $\mathbb{R}^d$ by extending it outside $\epsilon^{-1}U^\delta$ by zero. We can write
    \begin{align*}
    	\mathcal{B}_1^{\delta,\epsilon} g^\epsilon(\xi)= \int_{\mathbb{R}^d} g(x)e^{-ix\cdot\xi}\overline{\tilde{\phi}_1^{\delta,\epsilon}}(x;0)\,dx+\int_{\mathbb{R}^d} g(x)e^{-ix\cdot\xi}\left(\tilde{\phi}_1^\delta\left(\frac{x}{\epsilon};\epsilon\xi\right) -\tilde{\phi}^\delta_1\left(\frac{x}{\epsilon};0\right) \right)dx.
    \end{align*}
The first term above converges to the Fourier transform of $g$ on account of the normalization of $\phi_1(y;0)$ whereas the second term goes to zero since it is $O(\epsilon\xi)$ due to the Lipschitz continuity of the first regularized Bloch eigenfunction. More details including the proof of Lipschitz continuity of Bloch eigenvalues and eigenfunctions may be found in~\cite{Conca1997}.\end{proof}

\section{Homogenization theorem}\label{homotheorem}
For $\Omega\subseteq\mathbb{R}^d$, consider the boundary value problem
\begin{align}\label{tobehomo}
	\mathcal{A}^\epsilon u^\epsilon=-\nabla_x\cdot(A(x/\epsilon)\nabla_x u^\epsilon(x))=f\mbox{ in }\Omega\nonumber\\
	u^\epsilon=0 \mbox{ on }\partial\Omega
\end{align} where $f\in L^2(\Omega)$ and $\Omega$ is bounded domain with $C^2$ boundary. The hypothesis on $f$ and $\Omega$ is required to make sure that $u^\epsilon\in H^2(\Omega)$. In this section, we shall assume summation over repeated indices for ease of notation. We shall prove the following theorem.

\begin{theorem}\label{homog}
	Let $\Omega$ be a bounded domain in $\mathbb{R}^d$ with $C^2$ boundary and $f\in L^2(\Omega)$. Let $u^\epsilon\in H^1_0(\Omega)$ be a solution of~\eqref{tobehomo} such that $u^\epsilon$ converges weakly to $u^*$ in $H^1_0(\Omega)$, and
	\begin{align}\label{equation}
	\mathcal{A}^\epsilon u^\epsilon=f\,\mbox{in}\,~\Omega.
	\end{align}
	Then
	\begin{enumerate}
		\item For all $k=1,2,\ldots,d$, we have the following convergence of fluxes:
		\begin{align}
		a^\epsilon_{kl}(x)\frac{\partial u^\epsilon}{\partial x_l}(x)\rightharpoonup q_{kl}^*\frac{\partial u^*}{\partial x_l}(x) \mbox{ in } L^2(\Omega)\mbox{-weak}.
		\end{align}
		\item The limit $u^*$ satisfies the homogenized equation:
		\begin{align}\label{homoperator}
		\mathcal{A}^{hom}u^*=-\frac{\partial}{\partial x_k}\left(q^*_{kl}\frac{\partial u^*}{\partial x_l}\right)=f\,\mbox{ in }\,\Omega.
		\end{align}
	\end{enumerate}
\end{theorem}

The proof of Theorem~\ref{homog} is divided into the following steps. We begin by localizing the equation~\eqref{equation} which is posed on $\Omega$, so that it is posed on $\mathbb{R}^d$. We take the quasiperiodic Bloch transform $\mathcal{B}^{\delta,\epsilon}_1$ of this equation and pass to the limit $\epsilon\to 0$, followed by the limit $\delta\to 0$. 

\paragraph{Step 1:} Let $\psi_0$ be a fixed smooth function supported in a compact set $K\subset\mathbb{R}^d$. Since $u^\epsilon$ satisfies $\mathcal{A}^\epsilon u^\epsilon=f$, $\psi_0 u^\epsilon$ satisfies
\begin{align}\label{local}
\mathcal{A}^{\epsilon}(\psi_0 u^\epsilon)(x)=\psi_0f(x)+g^\epsilon(x)+h^{\epsilon}(x)\,\mbox{ in }\,\mathbb{R}^d,
\end{align} where
\begin{align}
g^\epsilon(x)&\coloneqq-\frac{\partial \psi_0}{\partial x_k}(x)a^\epsilon_{kl}(x)\frac{\partial u^\epsilon}{\partial x_l}(x),\label{geps}\\
h^{\epsilon}(x)&\coloneqq-\frac{\partial}{\partial x_k}\left(\frac{\partial \psi_0}{\partial x_l}(x)a^{\epsilon}_{kl}(x)u^\epsilon(x)\right),\label{heps}
\end{align}

\paragraph{Step 2:} Taking the first Bloch transform of both sides of the equation~\eqref{local}, we obtain for $\xi\in\epsilon^{-1}U^\delta$ a.e.
\begin{align}\label{Bloch}
\mathcal{B}^{\delta,\epsilon}_1(\mathcal{A}^\epsilon(\psi_0 u^\epsilon))(\xi)=\mathcal{B}^{\delta,\epsilon}_1(\psi_0 f)(\xi)+\mathcal{B}^{\delta,\epsilon}_1g^\epsilon(\xi)+\mathcal{B}^{\delta,\epsilon}_1h^{\epsilon}(\xi).
\end{align}

\paragraph{Step 3:}\label{step3} Observe that $\psi_0 u^\epsilon\in H^2(\mathbb{R}^d)$. We have
\begin{align}\label{expansion}
	\mathcal{B}^{\delta,\epsilon}_1 (\mathcal{A}^\epsilon(\psi_0 u^\epsilon))&=\int_{\mathbb{R}^d}\mathcal{A}^\epsilon(\psi_0 u^\epsilon)(x)e^{-ix\cdot\xi}\overline{\tilde\phi^{\delta,\epsilon}_1}(x;\xi)\,dx\nonumber\\
	&=\int_{\mathbb{R}^d}(\psi_0 u^\epsilon)(x)\mathcal{A}^\epsilon(e^{-ix\cdot\xi}\overline{\tilde\phi^{\delta,\epsilon}_1}(x;\xi))\,dx\nonumber\\
	&=\lambda_1^{\delta,\epsilon}(\xi)\int_{\mathbb{R}^d}(\psi_0 u^\epsilon)(x)e^{-ix\cdot\xi}\overline{\tilde\phi^\epsilon_1}(x;\xi)\,dx+\sqrt{\delta}\int_{\mathbb{R}^d}(\psi_0 u^\epsilon)(x)e^{-ix\cdot\xi}\overline{\tilde{\beta}^{\delta,\epsilon}_1}(x;\xi)\,dx\nonumber\\
	&=\lambda_1^{\delta,\epsilon}(\xi)\mathcal{B}^\epsilon_1 (\psi_0 u^\epsilon)+\sqrt{\delta}\int_{\mathbb{R}^d}(\psi_0 u^\epsilon)(x)e^{-ix\cdot\xi}\overline{\tilde{\beta}^{\delta,\epsilon}_1}(x;\xi)\,dx
\end{align}

\paragraph{Step 4:} In this step, we shall obtain bounds for $\tilde{\beta}_1^{\delta,\epsilon}$. This is done by employing the analyticity of the first regularized Bloch eigenfunction in a neighborhood of $\eta=0$. Let us write
\begin{align*}
	\phi_1^\delta(y;\eta)=\phi^\delta_1(y;0)+ \eta_l\frac{\partial\phi_1^\delta}{\partial\eta_l}(y;0) +\gamma^\delta(y;\eta),
\end{align*} where $\gamma^\delta(y;0)=0$, $\frac{\partial\gamma^\delta}{\partial\eta_l}(y;0)=0$ and $\sqrt{\delta}\gamma^\delta(\cdot;\eta)=O(|\eta|^2)$ in $L^\infty(U^\delta;H^1_\sharp(Q))$ where the order is uniform in $\delta$ on account of~\eqref{coercivity}. Therefore, $\displaystyle\sqrt{\delta}\frac{\partial^2\gamma^\delta}{\partial y_k^2}(\cdot;\eta)=O(|\eta|^2)$ in $L^\infty(U^\delta;H^{-1}_\sharp(Q))$ where the order is uniform in $\delta$.
Now, \begin{align}\label{analyticbloch}
	\phi_1^{\delta,\epsilon}(y^{'};\xi)=\phi_1^{\delta}\left(\frac{y^{'}}{\epsilon};\epsilon\xi\right)=\phi^\delta_1\left(\frac{y^{'}}{\epsilon};0\right)+ \epsilon\xi_l\frac{\partial\phi_1^\delta}{\partial\eta_l}\left(\frac{y^{'}}{\epsilon};0\right)+\gamma^\delta\left(\frac{y^{'}}{\epsilon};\epsilon\xi\right).
\end{align} 
Let us define $\displaystyle\alpha_l^{\delta,\epsilon}(y^{'})\coloneqq \frac{\epsilon}{i\phi_1^\delta(y^{'}/\epsilon;0)}\frac{\partial\phi_1^\delta}{\partial\eta_l}\left(\frac{y^{'}}{\epsilon};0\right)$, then $\alpha^{\delta,\epsilon}_l(y^{'})\in H^1_\sharp(\epsilon Q)$ solves the cell problem at $\epsilon$-scale posed in $\epsilon Q$, i.e.,
	\begin{align}\label{epsiloncellproblem}
	-D_{y^{'}}\cdot B^\epsilon(y^{'})D_{y^{'}} \alpha^{\delta,\epsilon}_l-\delta\Delta_{y^{'}}\alpha^{\delta,\epsilon}_l= D_{y^{'}}\cdot B^\epsilon(y^{'})e_l,
	\end{align}
which provides the estimate \begin{align}||D_{y^{'}}\alpha^{\delta,\epsilon}_l||^2_{L^2_\sharp(\epsilon Q)}+\delta||\nabla_{y^{'}} \alpha^{\delta,\epsilon}_l||^2_{L^2_\sharp(\epsilon Q)}\leq C,\end{align} for some generic constant $C$ not depending on $\epsilon$ and $\delta$. Therefore, we get 
\begin{align}\label{estimate2}
	\left(\sqrt{\delta}\Delta_{y^{'}}\alpha^{\delta,\epsilon}_l\right)\mbox{is bounded uniformly in } H^{-1}_\sharp(\epsilon Q).
\end{align}

Differentiating the equation~\eqref{analyticbloch} with respect to $y^{'}$ twice, we obtain
\begin{align*}
	\frac{\partial^2\phi_1^{\delta,\epsilon}}{\partial y^{'2}_k}(y^{'},\xi)=\xi_l\epsilon\frac{\partial^2}{\partial y^{'2}_k}\frac{\partial\phi_1^{\delta,\epsilon}}{\partial\eta_l}\left(y^{'};0\right)+\epsilon^{-2}\frac{\partial^2\gamma^{\delta}}{\partial y^{2}_k}\left(\frac{y^{'}}{\epsilon};\epsilon\xi\right).
\end{align*}
For $\xi$ belonging to the set $\{\xi: \epsilon\xi\in U^\delta\mbox{ and }|\xi|\leq M \}$, we have 
\begin{align*}
	\sqrt{\delta}\left|\frac{\partial^2\gamma^\delta}{\partial y_k^2}(\cdot;\eta)\right|\leq C\epsilon^2 M^2.
\end{align*}\vspace{-0.5cm}\begin{align}\label{estimate1}\mbox{Therefore, }\left(\sqrt{\delta}\epsilon^{-2}\frac{\partial^2\gamma^\delta}{\partial y_k^2}(y^{'}/\epsilon;\epsilon\xi)\right)\mbox{ is bounded uniformly in } L^2_{\text{loc}}(\mathbb{R}^d_\xi;H^{-1}_\sharp(\epsilon Q)).
\end{align}
From~\eqref{estimate2} and~\eqref{estimate1}, we have $\beta_1^{\delta,\epsilon}(y^{'},\xi)=\sqrt{\delta}\xi_l i\phi_1^{\delta}\left(\frac{y^{'}}{\epsilon};0\right)\Delta_{y^{'}}\alpha^{\delta,\epsilon}_l+\frac{\sqrt{\delta}}{\epsilon^2}\sum_{k=1}^M\frac{\partial^2\gamma^\delta}{\partial y_k^2}\left(\frac{y^{'}}{\epsilon};\epsilon\xi\right)$  is bounded uniformly in $L^{2}_{\text{loc}}(\mathbb{R}^d_\xi;H^{-1}_\sharp(\epsilon Q))$. As a consequence,we obtain $\left(\tilde{\beta}_1^{\delta,\epsilon}\right)$ is bounded uniformly in 
$L^{2}_{\text{loc}}(\mathbb{R}^d_\xi;H^{-1}_{\text{loc}}(\mathbb{R}^d))$.

\paragraph{Step 5:} Now, we are ready to pass to the limit $\epsilon\to 0$ in the equation~\eqref{Bloch}. In view of equation~\eqref{expansion}, equation~\eqref{Bloch} becomes \begin{align}\label{Blochexpanded}
\lambda_1^{\delta,\epsilon}(\xi)\mathcal{B}^\epsilon_1 (\psi_0 u^\epsilon)+\sqrt{\delta}\int_{\mathbb{R}^d}(\psi_0 u^\epsilon)(x)e^{-ix\cdot\xi}\overline{\tilde{\beta}^{\delta,\epsilon}_1}(x;\xi)\,dx=\mathcal{B}^{\delta,\epsilon}_1(\psi_0 f)(\xi)+\mathcal{B}^{\delta,\epsilon}_1g^\epsilon(\xi)+\mathcal{B}^{\delta,\epsilon}_1h^{\epsilon}(\xi).
\end{align} Let us denote $\displaystyle\Upsilon^{\delta,\epsilon}(\xi)=\int_{\mathbb{R}^d}(\psi_0 u^\epsilon)(x)e^{-ix\cdot\xi}\overline{\tilde{\beta}^{\delta,\epsilon}_1}(x;\xi)\,dx$. Let $K_2$ be a compact subset of $\mathbb{R}^d_\xi$. From the previous step, we have $$||\Upsilon^{\delta,\epsilon}||_{L^2(K_2)}\lesssim ||\tilde{\beta}_1^{\delta,\epsilon}||_{L^2(K_2;H^{-1}(K))} $$

Hence, $\Upsilon^{\delta,\epsilon}$ is bounded in $L^2_{\text{loc}}(\mathbb{R}^d_\xi)$ independent of $\delta$ and $\epsilon$. Therefore, it converges weakly to $\Upsilon^\delta$ in $L^{2}_{\loc}(\mathbb{R}^d_\xi)$ for a subsequence. Once more, since the sequence $\Upsilon^{\delta,\epsilon}$ is bounded uniformly in $\delta$, the weak limit $\Upsilon^\delta$ is also bounded uniformly in $\delta$.

The proofs of convergences of all terms except the second term on LHS in~\eqref{Blochexpanded} follows the same lines as in~\cite{Conca1997}. Therefore, passing to the limit in~\eqref{Blochexpanded} as $\epsilon\to 0$ we obtain for $\xi\in\mathbb{R}^d$
\begin{align}\label{Bloch2}
\frac{1}{2}\frac{\partial^2\lambda^\delta_1}{\partial\eta_k\partial\eta_l}(0)\xi_k\xi_l\widehat{\psi_o u^*}(\xi)+\sqrt{\delta}\Upsilon^\delta(\xi)&=(\psi_0 f)^{\widehat{}}(\xi)-\left(\frac{\partial \psi_0}{\partial x_k}(x)\sigma^*_k(x)\right)^{\bf\widehat{}}(\xi)\nonumber\\
&\qquad-i\xi_k q_{kl}^{*} \left(\frac{\partial \psi_0}{\partial x_l}(x)u^*(x)\right)^{\bf\widehat{}}(\xi),
\end{align} where $\sigma_k^*$ is the weak limit of the flux $a^\epsilon_{kl}(x)\frac{\partial u^\epsilon}{\partial x_l}(x)$.

\paragraph{Step 6:} Now, we may pass to the limit in equation~\eqref{Bloch2} as $\delta\to 0$. Using Theorem~\ref{Hessian}, Lemma~\ref{convergenceofapproxhomo}, and the uniform in $\delta$ bound for $\Upsilon^\delta$, we obtain the following equation.

\begin{align}\label{Bloch3}
q^*_{kl}\xi_k\xi_l\widehat{\psi_o u^*}(\xi)=(\psi_0 f)^{\widehat{}}(\xi)-\left(\frac{\partial \psi_0}{\partial x_k}(x)\sigma^*_k(x)\right)^{\bf\widehat{}}(\xi)-i\xi_k q_{kl}^{*} \left(\frac{\partial \psi_0}{\partial x_l}(x)u^*(x)\right)^{\bf\widehat{}}(\xi),
\end{align} where $\sigma_k^*$ is the weak limit of the flux $a^\epsilon_{kl}(x)\frac{\partial u^\epsilon}{\partial x_l}(x)$.

The rest of the steps involving the identification of $\sigma_k^*$ and the homogenized equation are the same as in~\cite{Conca1997} and are therefore omitted.

%\section*{\refname}
\bibliographystyle{plain}
\bibliography{mylit}

\def\cprime{$'$}
\begin{thebibliography}{10}

\bibitem{Baumgartel1985}
H.~Baumg\"{a}rtel.
\newblock {\em Analytic perturbation theory for matrices and operators},
  volume~15 of {\em Operator Theory: Advances and Applications}.
\newblock Birkh\"{a}user Verlag, Basel, 1985.

\bibitem{Besicovitch55}
A.~S. Besicovitch.
\newblock {\em Almost periodic functions}.
\newblock Dover Publications, Inc., New York, 1955.

\bibitem{BirmanSuslina2003}
M~Birman and T.~A. Suslina.
\newblock Second order periodic differential operators. threshold properties
  and homogenization.
\newblock {\em St. Petersburg Mathematical Journal}, 15(5):639--714, 2004.

\bibitem{Blanc2015}
X.~Blanc, C.~Le~Bris, and P.-L. Lions.
\newblock Local profiles for elliptic problems at different scales: defects in,
  and interfaces between periodic structures.
\newblock {\em Comm. Partial Differential Equations}, 40(12):2173--2236, 2015.

\bibitem{Blot2001}
J.~Blot and D.~Pennequin.
\newblock Spaces of quasi-periodic functions and oscillations in differential
  equations.
\newblock volume~65, pages 83--113. 2001.
\newblock Special issue dedicated to Antonio Avantaggiati on the occasion of
  his 70th birthday.

\bibitem{Bohr47}
Harald Bohr.
\newblock {\em Almost {P}eriodic {F}unctions}.
\newblock Chelsea Publishing Company, New York, N.Y., 1947.

\bibitem{Conca1997}
C.~Conca and M.~Vanninathan.
\newblock Homogenization of periodic structures via bloch decomposition.
\newblock {\em SIAM Journal on Applied Mathematics}, 57(6):1639--1659, 1997.

\bibitem{dubois2012properties}
Jean-Marie Dubois.
\newblock Properties-and applications of quasicrystals and complex metallic
  alloys.
\newblock {\em Chemical Society Reviews}, 41(20):6760--6777, 2012.

\bibitem{GloriaHabibi2016}
Antoine Gloria and Zakaria Habibi.
\newblock Reduction in the resonance error in numerical homogenization {II}:
  {C}orrectors and extrapolation.
\newblock {\em Found. Comput. Math.}, 16(1):217--296, 2016.

\bibitem{Kato1995}
T.~Kato.
\newblock {\em Perturbation theory for linear operators}.
\newblock Classics in Mathematics. Springer-Verlag, Berlin, 1995.

\bibitem{Kozlov78}
S.~M. Kozlov.
\newblock Averaging of differential operators with almost periodic rapidly
  oscillating coefficients.
\newblock {\em Mat. Sb. (N.S.)}, 107(149)(2):199--217, 317, 1978.

\bibitem{Ladyzhenskaya68}
O.~A. Ladyzhenskaya and N.~N. Ural'tseva.
\newblock {\em Linear and quasilinear elliptic equations}.
\newblock Academic Press, New York-London, 1968.

\bibitem{Meyer1995}
Yves Meyer.
\newblock Quasicrystals, {D}iophantine approximation and algebraic numbers.
\newblock In {\em Beyond quasicrystals ({L}es {H}ouches, 1994)}, pages 3--16.
  Springer, Berlin, 1995.

\bibitem{Oleinik1982}
O.~A. Oleinik and V.~V. Zhikov.
\newblock On the homogenization of elliptic operators with almost-periodic
  coefficients.
\newblock {\em Rendiconti del Seminario Matematico e Fisico di Milano},
  52(1):149--166, Dec 1982.

\bibitem{Reed1978}
M.~Reed and B.~Simon.
\newblock {\em Methods of modern mathematical physics. {IV}. {A}nalysis of
  operators}.
\newblock Academic Press [Harcourt Brace Jovanovich, Publishers], New
  York-London, 1978.

\bibitem{Schechtman84}
D.~Shechtman, I.~Blech, D.~Gratias, and J.~W. Cahn.
\newblock Metallic phase with long-range orientational order and no
  translational symmetry.
\newblock {\em Phys. Rev. Lett.}, 53:1951--1953, 1984.

\bibitem{Shubin78}
M~A Shubin.
\newblock Almost periodic functions and partial differential operators.
\newblock {\em Russian Mathematical Surveys}, 33(2):1, 1978.

\bibitem{Wellander2018}
Niklas Wellander, Sebasti\'{e}n Guenneau, and Elena Cherkaev.
\newblock Two-scale cut-and-projection convergence; homogenization of
  quasiperiodic structures.
\newblock {\em Math. Methods Appl. Sci.}, 41(3):1101--1106, 2018.

\end{thebibliography}
\end{document}